\documentclass[12pt,a4paper,twoside]{article}
\usepackage{amssymb,amsthm}
\usepackage[namelimits,sumlimits,fleqn]{amsmath}
\usepackage{setspace, bm}
\usepackage[a4paper,top=33mm, bottom=27mm, left=20mm, right=20mm]{geometry}
\usepackage[small, margin=20pt]{caption}
\usepackage{float}\restylefloat{figure}
\usepackage{url}
\usepackage[final]{graphicx}
\usepackage{subfig,multicol}
\usepackage{hyperref}
\usepackage[usenames,dvipsnames]{color}
\definecolor{darkblue}{rgb}{0.0,0.0,0.6}
\hypersetup{colorlinks=true,breaklinks=true,
            linkcolor=darkblue,urlcolor=darkblue,
            anchorcolor=darkblue,citecolor=darkblue}

\allowdisplaybreaks[4]

\newtheorem{theorem}{Theorem}[section]
\newtheorem{corollary}[theorem]{Corollary}
\newtheorem{lemma}[theorem]{Lemma}
\newtheorem{proposition}[theorem]{Proposition}
\newtheorem{indexedexample}[theorem]{Example}

\theoremstyle{definition}
\newtheorem{indexedclaim}[theorem]{Claim}
\newtheorem{definition}[theorem]{Definition}

\newtheorem*{note}{Note}

\numberwithin{equation}{section}

\usepackage{fancyhdr}
\title{The Cardinality of Sumsets: Different Summands}
\author{Brendan Murphy, Eyvindur Ari Palsson and Giorgis Petridis}
\date{}

\begin{document}

\pagenumbering{arabic}

\setcounter{section}{0}

\bibliographystyle{plain}

\onehalfspace

\thispagestyle{empty}

\maketitle

\begin{abstract}
We offer a compete answer to the following question on the growth of sumsets in commutative groups. Let $h$ be a positive integer and $A, B_1, B_2, \dots, B_h$ be finite sets in a commutative group. We bound $|A+B_1+\dots+B_h|$ from above in terms of $|A|$, $|A+B_1|, \dots,|A+B_h|$ and $h$. Extremal examples, which demonstrate that the bound is asymptotically sharp in all  parameters, are furthermore provided.
\end{abstract}

\section{Introduction}

Given (non-empty) finite sets $A, B_1, B_2, \dots, B_h$ in a commutative group, their \emph{sumset}  (also referred to as their \emph{Minkowski sum}) is
\[
A+B_1+\dots+B_h=\{a+b_1+\dots+b_h : a\in A, b_i\in B_i \mbox{ for $1\leq i\leq h$}\}.
\] 

We obtain an upper bound on the cardinality of $A+B_1+\dots+B_h$ in terms of $h$ and the cardinalities of $A$ and $A+B_1,\dots, A+B_h$. Note that the question becomes trivial unless some constraints are put on the sets as $|A+B_1+\dots+B_h| \leq |A| |B_1|\dots|B_h|$; and the bound is attained when $A, B_1,\dots, B_h$ are sets of distinct generators of a free commutative group. 

The best known upper bound is as follows.
\begin{theorem}\label{Existing}
Let $h$ and $m$ be positive integers and $\alpha_1,\dots,\alpha_h$ be positive real numbers. Suppose that $A, B_1,\dots,B_h$ are finite sets in a commutative group that satisfy $|A|=m$ and $|A+B_i| \leq \alpha_i m$ for all $1\leq i \leq h$. Then
\[
|A+B_1+\dots+B_h| \leq \alpha_1\dots\alpha_h   m^{2-1/h}.
\]
\end{theorem}

Theorem \ref{Existing} can be proved by different methods. It can be deduced from the work of Ruzsa in \cite{Ruzsa2006,Ruzsa2009}. It also follows by combining an inequality of Balister and Bollob\'as in \cite{Balister-Bollobas2007} with an inequality of Ruzsa \cite{Ruzsa1989}. Madiman, Marcus and Tetali have given a different proof of the inequality of Balister and Bollob\'as in \cite{MMT2008}. We discuss the various proofs in more detail in Section \ref{History}. It is worth pointing out here that the methods used by the three groups of authors are  different: Ruzsa relied on graph theory; Bollob\'as and Balister on projections; and Madiman, Marcus and Tetali on entropy.  

The upper bound in Theorem \ref{Existing} has the correct dependence on $\alpha$ and $m$. The following example (a modification of similar examples given by Ruzsa in \cite{Ruzsa2006,Ruzsa2009}) demonstrates this.

\begin{indexedexample} \label{Example}
Let $h$ be a positive integer. There exist infinitely many $(\alpha_1,\dots,\alpha_h)\in (\mathbb{Q}^+)^h$ with the following property. For each such $h$-tuple $(\alpha_1,\dots,\alpha_h)$ there exist infinitely many $m$ and sets $A, B_1,\dots,B_h$ in a commutative group with $|A|=m$, $|A+B_i|\leq (1+o(1))  \alpha_i m$ and 
\[
|A+B_1+\dots+B_h| \geq (1+o(1))  \frac{\left( 1- \frac{1}{h} \right)^{h-1}}{h}   \alpha_1\dots\alpha_h  m^{2-1/h}.
\]
The $o(1)$ term is $o_{m\rightarrow\infty}(1)$. 
\end{indexedexample}

We show that the sets in Example \ref{Example} are extremal to this problem by proving a matching upper bound and so settle the question of bounding from above the cardinality of higher sumsets in commutative groups. 

\begin{theorem}\label{Main}
Let $h$ be a positive integer, $\alpha_1,\dots,\alpha_h$ be positive real numbers and $m$ an arbitrarily large integer. Suppose that $A, B_1,\dots,B_h$ are finite sets in a commutative group that satisfy $|A|=m$ and $|A+B_i| \leq \alpha_i m$ for all $1\leq i \leq h$. Then
\begin{align*}
|A+B_1+\dots+B_h|   & \leq  \frac{ \left(1 - \frac{1}{h}\right)^{h-1} }{h}  \alpha_1\dots\alpha_h  \left( m^{2-1/h} + O\left(m^{2-2/h} \right) \right) \\
				& =  \left(1 + o(1)\right) \frac{ \left(1 - \frac{1}{h}\right)^{h-1} }{h}  \alpha_1\dots\alpha_h m^{2-1/h}.
\end{align*}
The $o(1)$ term is $o_{m\rightarrow\infty}(1)$. 
\end{theorem}

\begin{note}
For large $h$ the main term is roughly $ \displaystyle \frac{e^{-1}}{h}   \alpha_1\dots\alpha_h  m^{2-1/h}.$
\end{note}

The proof is a refinement of Ruzsa's graph theoretic approach. The upper bound in Theorem \ref{Main} is submultiplicative with respect to direct products. In other words if one replaces $A$ by, say, its Cartesian product $A\times A = \{(a,a^{\prime}) : a,a^{\prime}\in A\}$ and the $B_i$ by their Cartesian products $B_i\times B_i$, then (after some standard calculations of the form $|(A\times A)+(B\times B)| = |(A+B)\times (A+B)| = |A+B|^2$) one obtains
\[
\frac{\alpha_1\dots\alpha_h}{\sqrt{h}}  |A|^{2-1/h},
\]
which is weaker than what the theorem gives. This particular feature of the upper bound makes using one of the key ingredients in Ruzsa's method, the product trick, more delicate. From a technical point of view this is the greatest difficulty that must be overcome. 

The special case when $B_1=B_2=\dots=B_h$ and $\alpha_1=\dots=\alpha_h=\alpha$ was considered in \cite{GPCaSu}. The sumset $A+B_1+\dots+B_h$ in this case is abbreviated to $A+hB$. The upper bound obtained there is slightly stronger:
\begin{equation}\label{A+hB}
|A+hB| \leq (1+o(1))  \frac{C}{h^2} \alpha^h  m^{2-1/h}
\end{equation}
for an absolute constant $C>0$. The extra factor of $h$ in the denominator can be accounted for by the fact that while $|S_1+\dots+S_h| \leq |S_1| \dots |S_h|$ holds for general sets $S_i$, when the same set $S$ is added to itself one has the stronger inequality $|S+\dots+S|\leq \binom{|S|+h-1}{h}.$ Inequality \eqref{A+hB} probably does not have the correct dependence in $h$ as the largest value of $|A+hB|$ in examples is of the order $h^{-h-1} \alpha^h m^{2-1/h}$. It would be of interest to bridge that gap.

The proof of Theorem \ref{Main} is similar to that of inequality \eqref{A+hB}. There are nonetheless technical differences. Roughly speaking we combine ideas from the proof of inequality \eqref{A+hB} with a strategy used repeatedly in the literature (for example in \cite{GMR2008,Ruzsa2010}) to prove a generalisation of the afore mentioned result of Ruzsa. We could not find a result general enough for our purposes in the literature and so give a detailed proof in Section \ref{Graphs}.      
 
The paper is organised as follows. In Section \ref{History} we discuss the different proofs of Theorem \ref{Existing}. The proof of Theorem \ref{Main} is done in Section \ref{Proof of main}. Example \ref{Example} is described in Section \ref{Examples}. In Section \ref{Graphs} the graph theoretic framework of the proof is developed.

\section[Proof of Theorem~1.1]{Proof of Theorem \ref{Existing}}
\label{History}

Theorem \ref{Existing} follows by combining an inequality of Balister and Bollob\'as with an inequality of Ruzsa.

\begin{theorem}[Balister-Bollob\'as,~\cite{Balister-Bollobas2007}]\label{BalisterBollobas}
Let $h$ and $m$ be positive integers and $\alpha_1,\dots,\alpha_h$ be positive real numbers. Suppose that $A, B_1,\dots,B_h$ are finite sets in a commutative group that satisfy $|A|=m$ and $|A+B_i| \leq \alpha_i m$ for all $1\leq i \leq h$. Then for any subset $C\subseteq B_1+\dots+B_h$
\[
|A+C| \leq (\alpha_1\dots\alpha_h)^{1/h}   m   |C|^{1-1/h}.
\]
\end{theorem}

The proof given by Balister and Bollob\'as is  short and elegant. It combines an idea of Gyarmati, Matolcsi and Ruzsa in \cite{GMR2010} with the Box Theorem in \cite{Bollobas-Thomason1995}. Madiman, Marcus and Tetali gave a somewhat different proof based on entropy \cite{MMT2008}. The theorem can also be proved by methods developed by Ruzsa (for example in \cite{Ruzsa2006,Ruzsa2009}). 

To deduce Theorem \ref{Existing} one naturally sets $C=B_1+\dots+B_h$. This gives 
\begin{equation}\label{ExistingInt}
|A+B_1+\dots+B_h| \leq (\alpha_1\dots\alpha_h)^{1/h}   m   |B_1+\dots+B_h|^{1-1/h}.
\end{equation}

We are left with bounding $|B_1+\dots+B_h|$ in terms of $m$ and the $\alpha_i$. Ruzsa achieved this by modifying a graph theoretic method of Pl\"unnecke in \cite{Plunnecke1970}, a variant of which we describe in Section \ref{Graphs}.

\begin{theorem}[Ruzsa,~\cite{Ruzsa1989}]\label{Ruzsa}
Let $h$ and $m$ be positive integers and $\alpha_1,\dots,\alpha_h$ be positive real numbers. Suppose that $A, B_1,\dots,B_h$ are finite sets in a commutative group that satisfy $|A|=m$ and $|A+B_i| \leq \alpha_i m$ for all $1\leq i \leq h$. Then there exists a non-empty subset $\emptyset \neq X\subseteq A$ such that
\[
|X+B_1+\dots+B_h| \leq \alpha_1\dots\alpha_h   |X|.
\]
In particular 
\[
|B_1+\dots+B_h| \leq |X+B_1+\dots+B_h| \leq \alpha_1\dots\alpha_h  |X| \leq \alpha_1\dots\alpha_h  m.
\]
\end{theorem}
Substituting the last inequality in inequality \eqref{ExistingInt} gives the bound in Theorem \ref{Existing}.

Theorems \ref{Existing} and \ref{Ruzsa} differ crucially in the exponent of $m$. Ruzsa has shown in \cite{Ruzsa2010} that if one is interested in bounding $|X+B_1+\dots+B_h|$ for a suitably chosen large subset of $A$, then the correct exponent of $m$ is 1. 

Specifically he proved that for any $\varepsilon>0$ there exists a non-empty subset $\emptyset\neq X \subseteq A$ such that $|X|>(1-\varepsilon) |A|$ and\[
|X+B_1+\dots+B_h| \leq \left(\frac{h   \varepsilon^{1-h}-1}{h-1}\right) \alpha_1\dots\alpha_h  |X| \leq 2 \varepsilon^{1-h}  \alpha_1\dots\alpha_h  |X|.
\]

The exponent of $|X|$ in the upper bound remains 1 even when $X$ is required to have very large density in $A$. The nature of the upper bound changes when the cardinality of the whole of $A+B_1+\dots+B_h$ is bounded.

\section[Proof of Theorem~1.3]{Proof of Theorem \ref{Main}}
\label{Proof of main}

The upper bound in Theorem \ref{Main} is an increasing function of the $\alpha_i$ and the ratios $|A+B_i| /|A|$ are rational numbers so we may assume that $ \alpha_i \in \mathbb{Q}^+$. 

The next step is to reduce to the special case where all the $\alpha_i$ are equal. We prove the following.
\begin{proposition}\label{Main equal}
Let $h$ be a positive integer, $\alpha$ be a positive rational number and $m$ an arbitrarily large integer. Suppose that $A, B_1,\dots,B_h$ are finite sets in a commutative group that satisfy $|A|=m$ and $|A+B_i| \leq \alpha m$ for all $1\leq i \leq h$. Then
\begin{align*}
|A+B_1+\dots+B_h|  & \leq \alpha^h m + \frac{ \left(1 - \frac{1}{h}\right)^{h-1} }{h}  \alpha^h \left( m^{2-1/h} + O\left( m^{2-2/h} \right) \right) \\
				& = \frac{ \left(1 - \frac{1}{h}\right)^{h-1} }{h}  \alpha^h \left( m^{2-1/h} + O\left( m^{2-2/h} \right) \right) .
\end{align*}
\end{proposition}

Theorem \ref{Main} is deduced from the above proposition in a standard way by working in direct products of groups (see for example \cite{Ruzsa2009,Ruzsa2010}). 

\begin{proof}[Deduction of Theorem \ref{Main} from Proposition \ref{Main equal}]
Let $\alpha_i = p_i /q_i$ and set $n = q_1 \dots q_h.$ Furthermore, let $T_1,\ldots, T_h$ be pairwise disjoint sets of generators of a free abelian group $F$ with cardinality $n_i:=|T_i| = n \prod_{j\neq i} \alpha_j $; and let $0$ denote the identity of $F$. Each $n_i$  is chosen so that $\alpha_i  n_i$  is equal to $ n\left (\prod_j \alpha_j\right)$. 

We apply Proposition \ref{Main equal} to the sets $A^\prime = A \times \{0\}$, $B_1^\prime = B_1 \times T_1,\dots , B_h^\prime = B_h \times T_h$. As 
\begin{align*}
 |A^\prime+ B_i^\prime | = |T_i|   |A+B_i| \leq m  n \prod_{j=1}^h \alpha_j  = \left( n \prod_{j=1}^h \alpha_j \right) |A^\prime| 
\end{align*}
for all $i=1,\dots,h$ the proposition yields 
\begin{align*}
|A^\prime+B_1^\prime+\dots+B_h^\prime| & \leq \frac{\left( 1- \frac{1}{h} \right)^{h-1}}{h} \left( n \prod_{j=1}^h \alpha_j \right)^h \left(m^{2-1/h} + O(m^{2-2/h}) \right)\\
								 &   =  \frac{\left( 1- \frac{1}{h} \right)^{h-1}}{h} \left( \prod_{i=1}^h \alpha_i   n_i\right)   \left( m^{2-1/h} + O(m^{2-2/h})\right).
\end{align*}

Theorem \ref{Main} follows by observing that
\begin{align*}
|A^\prime+B_1^\prime+\dots+B_h^\prime| & =  |A + B_1+\dots+B_h|  |\{0\} + T_1 + \dots +T_h| \\
								  & =  |A + B_1+\dots+B_h|   \prod_{i=1}^h n_i \,,
\end{align*}
and dividing by $n_1\cdots n_h$.
\end{proof}

We next prove Proposition \ref{Main equal}. The rough strategy is as follows. We initially apply Theorem \ref{Ruzsa} to find a non-empty subset $\emptyset \neq X_1 \subseteq A$ whose growth under addition with the $B_i$ can be bounded. We are left with bounding 
\[
(A+B_1+\dots+B_h) \setminus (X_1+B_1+\dots+B_h).
\]
 
 We would like to iterate this process, which requires a stronger statement than Theorem \ref{Ruzsa}. From a technical point of view, this is the heart of the argument. It requires a detour in graph-theoretic techniques developed by Pl\"unnecke and Ruzsa and so is left for Section \ref{Graphs}. The key result we employ in the proof of Proposition \ref{Main equal} is as follows. It will be proved in a slightly stronger form as Corollary \ref{Restricted sumsets} in p.\pageref{Restricted sumsets}.

\begin{lemma}[Bound for sumsets with a component removed]\label{bound}
Let $h$ be a positive integer. Suppose that $A, B_1,\dots,B_h$ are finite sets in a commutative group and $E\subseteq A$ a subset of $A$.

If $\emptyset\neq X \subseteq A\setminus E$ is a subset of $A\setminus E$ that minimises the quantity
\[
\mu(Z) := \frac{1}{h} \sum_{i=1}^h \frac{ |(Z+B_i) \setminus (E+B_i)|}{|Z|}
\]
over all non-empty subsets $\emptyset \neq Z \subseteq A\setminus E$, then
\[
|(X+B_1+\dots+ B_h)   \setminus (E+B_1+\cdots + B_h)| \leq \mu^h   |X|,
\]
where
\[
\mu = \mu(X) = \frac{1}{h}  \sum_{i=1}^h \frac{ |(X+B_i)\setminus (E+B_i)|}{|X|} .
\]
\end{lemma}
Note that setting $E=\emptyset$ gives Theorem \ref{Ruzsa} for the special case when $\alpha_1=\dots=\alpha_h=\alpha$, as
\[
\mu = \min_{\emptyset\neq Z \subseteq A}  \frac{1}{h}  \sum_{i=1}^h \frac{ |Z+B_i|}{|Z|} \leq \frac{1}{h}  \sum_{i=1}^h \frac{ |A+B_i|}{|A|} \leq \alpha .
\]

The ultimate task for this section is to deduce Proposition \ref{Main equal} from the above estimate.
\begin{proof}[Proof of Proposition \ref{Main equal}]
Applying the bound stated above successively we partition $A$ into $X_1 \cup \dots \cup X_k$ for some finite $k$ ($A$ is finite), whose exact value is irrelevant to the argument. More precisely in the $j$th step set we set $E=\bigcup_{\ell=1}^{j-1} X_\ell$ ($E=\emptyset$ for $j=1$) and chose $X_j$ to be the \emph{minimal} non-empty subset of $A\setminus E$ that minimises the quantity
\[
\mu_j:=\frac{1}{h} \sum_{i=1}^h \frac{ |(Z+B_i) \setminus (E+B_i)|}{|Z|},
\]
which we set to be $\mu_j$. The inequality we get is
\begin{equation}\label{mu_j bound} 
\left| \left( X_j + B_1 + \dots + B_h \right) \setminus \left(\bigcup_{\ell = 1}^{j-1} X_\ell + B_1 +  \dots + B_h \right) \right| \leq \mu_j^h |X_j|. 
\end{equation}

It is crucial to observe that the defining properties (and especially the minimality) of the $X_j$ imply that the $\mu_j$ form an increasing sequence. Indeed $\mu_j \leq \mu_{j+1}$ as
\begin{align*}
\mu_j |X_j | + \mu_j | X_{j+1}| &   =     \mu_j |X_j \cup X_{j+1}| \\ 
				  	   & \leq   \frac{1}{h} \sum_{i=1}^h \left| \left((X_j \cup X_{j+1}) + B_i\right) \setminus \left(\bigcup_{\ell = 1}^{j-1} X_\ell + B_i \right) \right| \\ 
	   			  	   &   =     \frac{1}{h} \sum_{i=1}^h \left| (X_j + B_i) \setminus \left(\bigcup_{\ell = 1}^{j-1} X_\ell + B_i \right) \right| \\
	  			  	  &         ~~~~~ +~  \frac{1}{h} \sum_{i=1}^h \left| (X_{j+1} + B_i) \setminus \left(\bigcup_{\ell = 1}^{j} X_\ell + B_i \right) \right| \\
	  			  	   &  =    \mu_j |X_j| + \mu_{j+1} |X_{j+1}|.
\end{align*}

When the $\mu_j$ are large it turns out that replacing the estimate in \eqref{mu_j bound} by a more elementary one is more economical. We have 
\begin{align*}
\left| \left( X_j + B_1 + \dots + B_h \right) \setminus \left(\bigcup_{\ell = 1}^{j-1} X_\ell + B_1 +  \dots + B_h \right) \right|   & \leq   |X_j + B_1+\dots + B_h| \\ &\leq  |X_j|   |B_1+\dots+B_h| .
\end{align*}

To bound $|B_1+\dots+B_h|$ we adapt accordingly the argument in Theorem~\ref{Ruzsa}. 
\[ 
|B_1+\dots+B_h| \leq |X_1+B_1+\dots+B_h| \leq\mu_1^h |X| \leq \mu_1^h   m .
\]
Combining \eqref{mu_j bound} with the last two inequalities gives 
\[ 
\left| \left( X_j + B_1 + \dots + B_h \right) \setminus \left(\bigcup_{\ell = 1}^{j-1} X_\ell + B_1 +  \dots + B_h \right) \right| \leq   \min\left\{ \mu_j^h , \mu_1^h   m \right\} |X_j|. 
\]
Summing over $j=1,\dots,k$ gives
\begin{align*} 
|A+B_1+\dots+B_h|  &   =    \sum_{j=1}^k \left| \left( X_j + B_1 + \dots + B_h \right) \setminus \left(\bigcup_{\ell = 1}^{j-1} X_\ell + B_1 +  \dots + B_h \right) \right| \\ 
				& \leq  \sum_{j=1}^k  \min\left\{ \mu_j^h , \mu_1^h   m \right\} |X_j|.
\end{align*}

We are left with bounding the  sum
$
 \sum_{j=1}^k  \min \left\{ \mu_j^h  ,  \mu_1^h   m \right\}   |X_j|
$
subject to two constraints:
\[
\sum_{j=1}^k |X_j| = m 
\]
and
\begin{align*}
\sum_{j=1}^k \mu_j    |X_j|  &  =    \sum_{j=1}^k \left( \frac{1}{h} \sum_{i=1}^h \left| (X_j + B_i) \setminus \left( \bigcup_{\ell = 1}^{j-1} X_\ell + B_i \right)  \right|  \right) \\
					   &   =    \frac{1}{h} \sum_{i=1}^h \left( \sum_{j=1}^k \left| (X_j + B_i) \setminus \left( \bigcup_{\ell = 1}^{j-1} X_\ell + B_i \right)  \right|  \right)\\
				            &   =    \frac{1}{h} \sum_{i=1}^h |A+B_i| \\
				            & \leq  \alpha m. 
\end{align*}
The two quantities inside the $\min$ are equal if 
\[
\mu_j = \mu_* := \mu_1   m^{1/h}.
\]
As $\mu_j \geq \mu_1$ for all $1\leq j \leq h$, we can replace the $\min$ by the straight line
\[
\mu_1^h + (\mu_j - \mu_1) \frac{\mu_*^h - \mu_1^h}{\mu_*-\mu_1},
\]
which, thought of as function of $\mu_j$, intersects the curve $\mu_j^h$ at $\mu_j = \mu_1$ and $\mu_j = \mu_*.$ The slope is bounded by
\[
 \mu_1^{h-1} \frac{m - 1}{m^{1/h}-1} \leq \mu_1^{h-1} \left( m^{1-1/h} + 2 m^{1-2/h} \right).
\]

Therefore
\begin{align*} 
|A+B_1+\dots+B_h|  & \leq  \sum_{j=1}^k \left(\mu_1^h +  \mu_1^{h-1} (\mu_j - \mu_1) \left( m^{1-1/h} + 2 m^{1-2/h} \right)\right)   |X_j| \\ 
				&  =     \mu_1^h \sum_{j=1}^k |X_j|  +  \mu_1^{h-1} \left( m^{1-1/h} + 2 m^{1-2/h} \right)\sum_{j=1}^k \mu_j |X_j| \\
				& ~~~~~ -  ~\mu_1^{h} \left( m^{1-1/h} + 2 m^{1-2/h} \right)\sum_{j=1}^k |X_j| \\
				& \leq   \mu_1^h  m + (\alpha - \mu_1)  \mu_1^{h-1}  \left( m^{2-1/h} + 2 m^{2-2/h} \right) \\
				& \leq   \alpha^h  m + (\alpha - \mu_1)  \mu_1^{h-1}   \left( m^{2-1/h} + 2 m^{2-2/h} \right).
\end{align*}

The final task is to select the value of $1\leq \mu_1 \leq \alpha$ that maximises this expression. Differentiating the expression $(\alpha-\mu_1)\mu_1^{h-1}$ with respect to $\mu_1$ gives that it is maximised when
\[ 
(h-1)   (\alpha - \mu_1) = \mu_1 \implies \mu_1 = \left(1 - \frac{1}{h}\right) \alpha \mbox{ or } \alpha-\mu_1 = \frac{\alpha}{h}.
\]
Substituting above gives 
\begin{align*}
|A+B_1+\dots+B_h|  & \leq \alpha^h  m + \frac{ \left(1 - \frac{1}{h}\right)^{h-1} }{h}  \alpha^h  \left( m^{2-1/h} + 2 m^{2-2/h} \right) \\ 
				& = \alpha^h m + \frac{ \left(1 - \frac{1}{h}\right)^{h-1} }{h} \alpha^h  \left( m^{2-1/h} + O\left(m^{2-2/h} \right) \right).
\end{align*}
\end{proof}

This completes the proof of Proposition \ref{Main equal} modulo the proof of Lemma~\ref{bound} on p.~\pageref{bound}, which as we have seen implies Theorem \ref{Main}. The proof of the estimate is given in Section \ref{Graphs}. We next provide examples which show that the upper bound given by Theorem \ref{Main} is asymptotically sharp.

\section{Examples}
\label{Examples}

We construct the sets in Example \ref{Example}. To keep the notation simple we assume that the $\alpha_i$ are all equal: $\alpha_1 = \dots = \alpha_h =  \alpha.$ 

Once these examples have been constructed, it is straightforward to construct examples for different $(\alpha_1,\dots,\alpha_h)$ by considering direct products. Very much like in the first step of the proof of Theorem \ref{Main} in Section \ref{Proof of main} one then considers sets $A^{\prime} = A \times \{0\}, B_1^{\prime} = B_1 \times T_1,\dots, B_h^{\prime}= B_h \times T_h$ to get a different $h$-tuple $(\alpha_1,\dots,\alpha_h)$, where $\alpha_i = \alpha   |T_i|$. The $T_i$ are sets of distinct generators of a free commutative group. The details are as follows,
\[
\frac{|A^{\prime}+B_i^{\prime}|}{|A^{\prime}|} = \frac{|(A+B_i) \times T_i|}{|A|} = \frac{|A+B_i|}{|A|}   |T_i| \leq (1+o(1))  \alpha |T_i|
\]
and 
\begin{align*}
|A^{\prime}+B_1^{\prime}+\dots+B_h^{\prime}|  &   =       |(A+B_1+\dots+B_h) \times (T_1+\dots + T_h)|\\ 
				  &   =       |A+B_1+\dots+B_h| |T_1+\dots + T_h|\\ 
				  & \geq  (1+o(1))  \alpha^h |A|^{2-1/h}  |T_1|\cdots |T_h|\\
				  &   =     (1+o(1))  \alpha_1\cdots\alpha_h |A^{\prime}|^{2-1/h}.
\end{align*}

To construct the sets for the special case when $\alpha_i=\alpha$ for all $i$, we fix $h$ and let $a$ and $l$ be integers, which we consider as variables with $a$ assumed to be large and divisible by $h-1$. We set $b=l a$ and work in $\mathbb{Z}_b^k$, where $k= h + a^{h-1}/(h-1)$. We write $x_i$ for the $i$th coordinate of the vector $x$.

We consider $A = A_1 \cup A_2$ where 
\[ 
A_1= \{x : x_i\in\{0,l,2l,\dots,(a-1)l\}~\mbox{for}~1\leq i\leq h~\mbox{and}~x_i=0~\mbox{otherwise}\}
\]
and $A_2$ is a collection of $a^{h-1}/(h-1)$ independent points 
\[ 
A_2= \bigcup_{j=h+1}^k\{x : x_j = 1  ,\; x_j=0 \mbox{ otherwise}\}.
\]
$B_i$ is taken to be a copy of $\mathbb{Z}_b$ 
\[ 
B_i = \{x : x_i\in \{0,\dots,b-1\} ,\; x_j=0\mbox{ for all } j\neq i \}.
\]
We now estimate the cardinality of the sets that interest us.
\[ 
|A|= |A_1| + |A_2| = a^h + \frac{a^{h-1}}{h-1} =(1+o(1)) a^h.
\]
As $h$ is fixed different values of $a$ result to different values of $m$. 

To bound $|A+B_i|$ we note that $|A_1 + B_i|$ equals
\[
| \{x : x_i\in \{0,\dots,b-1\} , x_j \in \{0,\ell, 2\ell, \dots , (a-1) \ell\}, j\neq i  \}|  = b a^{h-1}
\]
and that $\displaystyle |A_2+B_i| =  |A_2| |B_i|= \frac{b a^{h-1}}{h-1}$. Thus
\begin{align*}
|A+B_i|  &\leq  |A_1+B_i|+|A_2+B_i|\\
       	     &\leq  b a^{h-1}+ \frac{b a^{h-1}}{h-1}\\
       	     & =   \left(1+\frac{1}{h-1} \right) l a^{h}\\
               & =    (1+o(1)) \left(1+\frac{1}{h-1}\right)  l  m.
\end{align*}
Therefore $\displaystyle  \alpha = \left(1+\frac{1}{h-1}\right) l$.

$h$ is fixed and so different values of $l$ result to different values of $\alpha$.

To bound $|A+B_1+\dots+B_h|$ from below observe that $|B_1 + \dots + B_h| = |\mathbb{Z}_b^h|=b^h$ and that for $a,a^\prime\in A_2$ the intersection $(a+B_1+\dots+B_h)\cap(a^\prime+B_1+\dots+B_h)$ is emty. Thus
\begin{align*}
|A+B_1+\dots+B_h|  &\geq  |A_2+B_1+\dots+B_h|\\
       				& =     |A_2|   |B_1+\dots+B_h|\\
			         & =    \frac{a^{h-1}}{h-1}   b^h \\
       				& =    \frac{ l^h }{h-1}   a^{2h-1}\\
			        & =     (1+o(1)) \frac{\left( 1 - \frac{1}{h}\right)^{h}}{h-1}  \alpha^h   m^{2-1/h}\\
			         & =     (1+o(1)) \frac{\left( 1 - \frac{1}{h}\right)^{h-1}}{h}  \alpha^h   m^{2-1/h}.
\end{align*}
We are done. As is expected the structure of the sets presented here is such that every inequality in Section \ref {Proof of main} is more or less attained.

\section{Graph Theory}
\label{Graphs}
In this section we develop the graph theoretic framework necessary for our proof of the estimate on p.\pageref{bound}; the last step of the proof of Theorem \ref{Main}. The results and methods of this section are influenced by the work of Ruzsa, c.f. \cite{Ruzsa2009,Ruzsa2010}. 

We define a type of layered commutative graph, called a \emph{commutative hypercube graph}, that generalises the addition graph associated to sumsets of the form $A+B_1+\cdots+B_h$, defined in the first example below. The class of commutative hypercube graphs includes graphs that result from removing a component from an addition graph. The main result of this section is an analog of Theorem \ref{Ruzsa} for commutative hypercube graphs.

Throughout this section $\biguplus$ stands for disjoint union.

\subsection{Hypercube graphs and their products}
\label{Hypercube}

Let $Q_h$ denote the set of all subsets of $\{1,\ldots,h\}$ and for $\ensuremath{I}$ in $Q_h$, let $|\ensuremath{I}|$ denote the cardinality of $\ensuremath{I}$.
Given $\ensuremath{I}$ and $\ensuremath{I}^{\prime}$ in $Q_h$, we will write $\ensuremath{I}\to\ensuremath{I}^{\prime}$ if $\ensuremath{I}^{\prime}=\ensuremath{I}\cup\{i\}$ for some $i\not\in\ensuremath{I}$.
\begin{definition}[Hypercube Graph]
\label{hypercube}
  Let $\ensuremath{\mathcal{G}}$ be a directed graph with vertex set $V$ and edge set $E$.
We say that $\ensuremath{\mathcal{G}}$ is a \emph{hypercube graph indexed by $Q_h$}\/ if it satisfies two conditions:
\begin{itemize}
\item[(i)] For each $\ensuremath{I}$ in $Q_h$ there exists a set $U_{\ensuremath{I}}\subseteq V$ such that $V$ is the disjoint union of the $U_{\ensuremath{I}}$s: $V=\biguplus_{\ensuremath{I}\in Q_h}U_{\ensuremath{I}}$. \label{hypercube vertices}
\item[(ii)] There is an edge $u\to v$ in $E$ only if $u\in U_{\ensuremath{I}}$ and $v\in U_{\ensuremath{I}^{\prime}}$ where $\ensuremath{I}\to\ensuremath{I}^{\prime}$. \label{hypercube edges}
\end{itemize}
\end{definition}
For short, we may say $\ensuremath{\mathcal{G}}$ is a ``$Q_h$-hypercube graph.''
Note that a $Q_h$-hypercube graph is a layered graph with $h+1$ layers: $V=V_0\cup\cdots\cup V_{h}$, where $V_i=\biguplus_{|\ensuremath{I}|=i}U_{\ensuremath{I}}$.

We give some examples of hypercube graphs.
The most important example of a hypercube graph is an addition graph with different summands, which are featured in \cite{Ruzsa2010}.
\begin{indexedexample}[Addition graphs]
  Let $A,B_1,\ldots,B_h$ be finite subsets of a commutative group $\ensuremath{G}$.
Their addition graph $\ensuremath{\mathcal{G}}_+(A,B_1,\ldots,B_h)$ is defined as follows: for each $\ensuremath{I}$ in $Q_h$, let $U_{\ensuremath{I}}=A+\sum_{i\in\ensuremath{I}}B_i$.
We consider each $U_{\ensuremath{I}}$ to be contained in a separate copy of $\ensuremath{G}$, and we let $V=\biguplus_{\ensuremath{I}\in Q_h}U_{\ensuremath{I}}$.
For each vertex $x$ in $U_{\ensuremath{I}}$ there is an edge to $y$ in $U_{\ensuremath{I}^{\prime}}$ if $\ensuremath{I}^{\prime}=\ensuremath{I}\cup\{i\}$ and $y=x+b$ for some $b$ in $B_i$.
Thus $\ensuremath{\mathcal{G}}_+(A,B_1,\ldots,B_h)$ is a hypercube graph indexed by $Q_h$.
\end{indexedexample}
Note that any subgraph of a $Q_h$-hypercube graph is automatically a $Q_h$-hypercube graph.
For certain induced subgraphs of a hypercube graph, we can say something more.
We recall from \cite{Ruzsa2009} a definition.
\begin{definition} [Channels of directed graphs]
Given a directed graph $\ensuremath{\mathcal{G}}=\ensuremath{\mathcal{G}}(V,E)$ and two sets of vertices $X,Y\subseteq V$, the \emph{channel $\ensuremath{\overline{\ensuremath{\mathcal{G}}}}(X,Y)$}\/ between $X$ and $Y$ is the subgraph of $\ensuremath{\mathcal{G}}$ induced by the set of vertices that lie on a path from $X$ to $Y$ (including endpoints).
\end{definition}
\begin{indexedexample}[Channels are hypercube graphs]\label{ChHyp}
Let $\ensuremath{\mathcal{G}}$ be a hypercube graph indexed by $Q_h$ and let $I$ and $\ensuremath{I}^{\prime}$ be elements of $Q_h$ such that $\ensuremath{I} \subseteq \ensuremath{I}^{\prime}$.
Given subsets $X\subseteq U_{\ensuremath{I}}$ and $Y\subseteq U_{\ensuremath{I}^{\prime}}$, the channel $\ensuremath{\overline{\ensuremath{\mathcal{G}}}}(X,Y)$ is a hypercube graph indexed by $Q_j$, where $j=|\ensuremath{I}^\prime\setminus\ensuremath{I}|$.
\end{indexedexample}

\begin{proof}
Since the edges of $\ensuremath{\overline{\ensuremath{\mathcal{G}}}}$ are edges of $\ensuremath{\mathcal{G}}$, condition \ref{hypercube edges} of Definition \ref{hypercube} is automatically satisfied.
Thus it remains to be shown that condition \ref{hypercube vertices} is satisfied.

Note that the set of $\ensuremath{J}$ in $Q_h$ such that $\ensuremath{I}\subseteq\ensuremath{J}\subseteq\ensuremath{I}^\prime$ is in one-to-one correspondence with $Q_j$.
Fixing one such correspondence, let $\bar{\ensuremath{J}}$ denote the element in $Q_j$ corresponding to $\ensuremath{J}$ and set $U_{\bar{\ensuremath{J}}}(\ensuremath{\overline{\ensuremath{\mathcal{G}}}})=V(\ensuremath{\overline{\ensuremath{\mathcal{G}}}})\cap U_{\ensuremath{J}}$.
Since any vertex in $V(\ensuremath{\overline{\ensuremath{\mathcal{G}}}})$ must be an element of some $U_{\ensuremath{J}}$ with $\ensuremath{I}\subseteq\ensuremath{J}\subseteq\ensuremath{I}^\prime$, we have $V(\ensuremath{\overline{\ensuremath{\mathcal{G}}}})=\biguplus_{\bar{\ensuremath{J}}\in Q_j}U_{\bar{\ensuremath{J}}}(\ensuremath{\overline{\ensuremath{\mathcal{G}}}})$, as desired.
\end{proof}

To prove the analog of Theorem \ref{Ruzsa}, we must define a type of graph product between hypercube graphs that is motivated by addition graphs of product sets.
\begin{definition}[Hypercube Product]
  Let $\ensuremath{\mathcal{G}}^\prime$ and $\ensuremath{\mathcal{G}}^{\prime\prime}$ be hypercube graphs indexed by $Q_h$.
We define a hypercube graph $\ensuremath{\mathcal{G}}=\ensuremath{\mathcal{G}}^\prime\otimes\ensuremath{\mathcal{G}}^{\prime\prime}$ also indexed by $Q_h$ as follows: for each $\ensuremath{I}\in Q_h$, we define $U_{\ensuremath{I}}(\ensuremath{\mathcal{G}})=U_{\ensuremath{I}}(\ensuremath{\mathcal{G}}^\prime)\times U_{\ensuremath{I}}(\ensuremath{\mathcal{G}}^{\prime\prime})$, and for $(u,v)\in U_{\ensuremath{I}}(\ensuremath{\mathcal{G}}),(u^\prime,v^\prime)\in U_{\ensuremath{I}^\prime}(\ensuremath{\mathcal{G}})$, we have $(u,v)\to(u^\prime,v^\prime)$ if and only if $u\to u^\prime$ and $v\to v^\prime$.
We call $\ensuremath{\mathcal{G}}$ the \emph{hypercube product}\/ of $\ensuremath{\mathcal{G}}^\prime$ and $\ensuremath{\mathcal{G}}^{\prime\prime}$.
\end{definition}
It is easy to see that $\ensuremath{\mathcal{G}}_+(A^\prime\times A^{\prime\prime},B_1^\prime\times B_1^{\prime\prime},\ldots,B_h^\prime\times B_h^{\prime\prime})$ is the hypercube product of $\ensuremath{\mathcal{G}}_+(A^\prime,B_1^\prime,\ldots,B_h^\prime)$ and $\ensuremath{\mathcal{G}}_+(A^{\prime\prime},B_1^{\prime\prime},\ldots,B_h^{\prime\prime})$. In this sense, direct products in the group setting correspond to hypercube products in the graph setting and so hypercube products are natural objects.

\subsection{Square commutativity}
\label{Square commutativity}

The key feature of addition graphs that makes them useful in additive number theory is that they capture in a graph theoretic way the commutativity of addition. This particular feature was first exploited by Pl\"unnecke in \cite{Plunnecke1970}, who worked with a class of directed layered graphs he called \emph{commutative}. The importance of commutative graphs to additive number theory is detailed in \cite{Nathanson1996,Tao-Vu2006,Ruzsa2009}. We will only mention them briefly as we need a stronger form of commutativity in order to prove Theorem \ref{Ruzsa}, one that works better for hypercube graphs. 

First we make an auxiliary definition: given index sets $\ensuremath{I},\ensuremath{I}^\prime,$ and $\ensuremath{I}^{\prime \prime}$ in $Q_h$ such that $\ensuremath{I}\to\ensuremath{I}^{\prime}\to\ensuremath{I}^{\prime \prime}$, there is a unique index set $\ensuremath{I}^{\prime}_c$ in $Q_h$ such that $\ensuremath{I}^{\prime}_c\not=\ensuremath{I}^{\prime}$ and $\ensuremath{I}\to\ensuremath{I}^{\prime}_c\to\ensuremath{I}^{\prime \prime}$; explicitly $\ensuremath{I}^{\prime}_c=\ensuremath{I}\cup(\ensuremath{I}^{\prime \prime}\setminus\ensuremath{I}^{\prime})$.
We will call $\ensuremath{I}^{\prime}_c$ the \emph{associate}\/ of $\ensuremath{I}^{\prime}$.
\begin{definition}[Square Commutativity]
  Let \ensuremath{\mathcal{G}}\ be a hypercube graph indexed by $Q_h$.
We say that \ensuremath{\mathcal{G}}\ is \emph{square commutative}\ if it satisfies two conditions:
\begin{enumerate}
\item Upward square commutativity: Given indices $\ensuremath{I},\ensuremath{I}^{\prime}$ and $\ensuremath{I}^{\prime \prime}$ in $Q_h$ such that $\ensuremath{I}\to\ensuremath{I}^{\prime}\to\ensuremath{I}^{\prime \prime}$, and vertices $v\in U_{\ensuremath{I}}$, $v^{\prime}\in U_{\ensuremath{I}^{\prime}}$, and $v_1^{\prime \prime},\ldots, v_n^{\prime \prime}\in U_{\ensuremath{I}^{\prime \prime}}$ such that $v\to v^{\prime}\to v_i^{\prime \prime}$ for $i=1,\ldots,n$, there exist distinct vertices $v_1^{\prime},\ldots,v_n^{\prime}\in U_{\ensuremath{I}^{\prime}_c}$ such that $v\to v_i^{\prime}\to v_i^{\prime \prime}$ for $i=1,\ldots,n$.
\item Downward square commutativity: Given indices $\ensuremath{I},\ensuremath{I}^{\prime}$ and $\ensuremath{I}^{\prime \prime}$ in $Q_h$ such that $\ensuremath{I}\to\ensuremath{I}^{\prime}\to\ensuremath{I}^{\prime \prime}$, and vertices $v_1,\ldots,v_n\in U_{\ensuremath{I}}$, $v^{\prime}\in U_{\ensuremath{I}^{\prime}}$, and $v^{\prime \prime}\in U_{\ensuremath{I}^{\prime \prime}}$ such that $v_i\to v^{\prime}\to v^{\prime \prime}$ for $i=1,\ldots,n$, there exist distinct vertices $v_1^{\prime},\ldots,v_n^{\prime}\in U_{\ensuremath{I}^{\prime}_c}$ such that $v_i\to v_i^{\prime}\to v^{\prime \prime}$ for $i=1,\ldots,n$.
\end{enumerate}
\end{definition}
Square commutative graphs are commutative in the sense defined by Pl\"unnecke; square commutativity strengthens commutativity by requiring that the alternate paths from $v$ to $v^{\prime \prime}_i$ (or from $v_i$ to $v^{\prime \prime}$) go through the associate vertex set.
This is an important observation as later on we will need to apply Pl\"unnecke's result.

In our language, Ruzsa has already shown in pp.~597-598 of \cite{Ruzsa2010} that addition graphs are square commutative:
\begin{proposition}[Ruzsa, \cite{Ruzsa2010}]
  Let $A,B_1,\ldots,B_h$ be subsets of a commutative group.
Then their addition graph $\ensuremath{\mathcal{G}}_+(A,B_1,\ldots,B_h)$ is square commutative.
\end{proposition}
Channels of square commutative hypercube graphs are also square commutative.  
\begin{lemma}\label{ChSqCom}
Let $\ensuremath{\mathcal{G}}$ be a square commutative hypercube graph indexed by $Q_h$, and let $\ensuremath{\overline{\ensuremath{\mathcal{G}}}}=\ensuremath{\overline{\ensuremath{\mathcal{G}}}}(X,Y)$ be a channel of $\ensuremath{\mathcal{G}}$.
Then $\ensuremath{\overline{\ensuremath{\mathcal{G}}}}$ is square commutative.
Additionally, if $X \subseteq U_{\ensuremath{I}}$ and $Y \subseteq U_{\ensuremath{I}^{\prime}}$ where $\ensuremath{I} \subsetneq \ensuremath{I}^{\prime}\in Q_h$, then $\ensuremath{\overline{\ensuremath{\mathcal{G}}}}$ is a $Q_j$ square commutative hypercube graph, where $j=|\ensuremath{I}^{\prime}\setminus\ensuremath{I}|$.
\end{lemma}
\begin{proof}
We have already shown on p.\pageref{ChHyp} that $\ensuremath{\overline{\ensuremath{\mathcal{G}}}}$ is a hypercube graph indexed by $Q_{|\ensuremath{I}^\prime \setminus \ensuremath{I}|}$. That it is square commutative follows from the fact that $\ensuremath{\mathcal{G}}$ is square commutative combined with the fact that if $x,z\in V(\ensuremath{\overline{\ensuremath{\mathcal{G}}}})$ and $x\to y\to z$ then $y\in V(\ensuremath{\overline{\ensuremath{\mathcal{G}}}})$.
\end{proof}
Now that we have shown that the main examples of hypercube graphs are square commutative, we will prove that square commutativity is inherited by products.
\begin{lemma}\label{SqComProd}
  Let $\ensuremath{\mathcal{G}}_1$ and $\ensuremath{\mathcal{G}}_2$ be square commutative hypercube graphs indexed by $Q_h$, and let $\ensuremath{\mathcal{G}}=\ensuremath{\mathcal{G}}_1\otimes\ensuremath{\mathcal{G}}_2$ be their hypercube product.
Then $\ensuremath{\mathcal{G}}$ is square commutative.
\end{lemma}
\begin{proof}
The proof is a straightforward verification of square commutativity.
We will prove only the upward condition, since the proof of the downward condition is similar.

Let $\ensuremath{I},\ensuremath{I}^{\prime},$ and $\ensuremath{I}^{\prime \prime}$ be indices in $Q_h$ such that $\ensuremath{I}\to\ensuremath{I}^{\prime}\to\ensuremath{I}^{\prime \prime}$, and suppose we have vertices $(u,v)\in U_{\ensuremath{I}}(\ensuremath{\mathcal{G}})$, $(u^{\prime},v^{\prime})\in U_{\ensuremath{I}^{\prime}}(\ensuremath{\mathcal{G}})$, and $(u^{\prime \prime}_1,v^{\prime \prime}_1),\ldots,(u^{\prime \prime}_n,v^{\prime \prime}_n)\in U_{\ensuremath{I}^{\prime \prime}}(\ensuremath{\mathcal{G}})$ such that $(u,v)\to (u^{\prime},v^{\prime})\to (u^{\prime \prime}_i,v^{\prime \prime}_i)$ for $i=1,\ldots,n$.
We must find $(u^{\prime}_1,v^{\prime}_1),\ldots,(u^{\prime}_n,v^{\prime}_n)\in U_{\ensuremath{I}^{\prime}_c}(\ensuremath{\mathcal{G}})$ such that $(u,v)\to (u^{\prime}_i,v^{\prime}_i)\to (u^{\prime \prime}_i,v^{\prime \prime}_i)$ for $i=1,\ldots,n$.

Consider the sequences of vertices $u\to u^{\prime}\to u^{\prime \prime}_i$ in $\ensuremath{\mathcal{G}}_1$.
Since $\ensuremath{\mathcal{G}}_1$ is square commutative, there exist distinct vertices $u^{\prime}_i\in U_{\ensuremath{I}^{\prime}_c}(\ensuremath{\mathcal{G}}_1)$ such that $u\to u^{\prime}_i\to u^{\prime \prime}_i$ for $i=1,\ldots,n$.
Similarly there exist distinct vertices $v^{\prime}_i\in U_{\ensuremath{I}^{\prime}_c}(\ensuremath{\mathcal{G}}_2)$ such that $v\to v^{\prime}_i\to v^{\prime \prime}_i$ for $i=1,\ldots,n$.
Thus we have distinct vertices $(u^{\prime}_i,v^{\prime}_i)\in U_{\ensuremath{I}^{\prime}_c}(\ensuremath{\mathcal{G}}_1)\times U_{\ensuremath{I}^{\prime}_c}(\ensuremath{\mathcal{G}}_2)=U_{\ensuremath{I}^{\prime}_c}(\ensuremath{\mathcal{G}})$ such that $(u,v)\to (u^{\prime}_i,v^{\prime}_i)\to (u^{\prime \prime}_i,v^{\prime \prime}_i)$ for $i=1,\ldots,n$, as desired.
\end{proof}

\subsection{A Pl\"unnecke-type inequality for square commutative graphs}
\label{Plunnecke-type}

The main goal in this section is to extend Ruzsa's Theorem \ref{Ruzsa}.
Our result can furthermore be thought of as an extension to square commutative graphs of Pl\"unnecke's inequality (Theorem \ref{Plunnecke} below). Before we do this we need to establish some notation and lemmas regarding magnification ratios.
 
Given a directed graph $\ensuremath{\mathcal{G}}$ and subsets $X,Y\subseteq V(\ensuremath{\mathcal{G}})$, we will use $\mathrm{Im}_\ensuremath{\mathcal{G}}(X,Y)$ to denote the set of elements in $Y$ that can be reached from $X$ by paths in $\ensuremath{\mathcal{G}}$.

If $\ensuremath{\mathcal{G}}$ has layers $V_0,\ldots, V_{h}$, we will use $\mu_i(\ensuremath{\mathcal{G}})$ to denote the \emph{$i$th magnification ratio} of $\ensuremath{\mathcal{G}}$, which is defined as
\[
\mu_i(\ensuremath{\mathcal{G}}):=\min_{\varnothing\not=Z\subseteq V_0}\frac{|\mathrm{Im}_\ensuremath{\mathcal{G}}(Z,V_i)|}{|Z|}.
\]
We will say that $\emptyset \neq X \subseteq V_0$ \emph{achieves} $\mu_i(\ensuremath{\mathcal{G}})$ when $\displaystyle \mu_i(\ensuremath{\mathcal{G}}) = \frac{|\mathrm{Im}_\ensuremath{\mathcal{G}}(X,V_i)|}{|X|}.$

Pl\"unnecke bounded the growth of magnification ratios of commutative graphs. We state a special case of his result that will be applied later.
\begin{theorem} [Pl\"unnecke, \cite{Plunnecke1970}]\label{Plunnecke}
Let $h\geq 1$ be a positive integer and $\ensuremath{\mathcal{G}}$ be a commutative graph. Then 
\[
\mu_h(\ensuremath{\mathcal{G}}) \leq \mu_1(\ensuremath{\mathcal{G}})^h.
\]
\end{theorem}
Square commutative graphs are commutative, so Theorem \ref{Plunnecke} applies; however, the bound on $\mu_h(\ensuremath{\mathcal{G}})$ is not adequate for our purpose. The goal of this subsection is to improve it for square commutative graphs.

If $\ensuremath{\mathcal{G}}$ is a hypercube graph indexed by $Q_h$, then the magnification of a subset $\emptyset\neq X\subseteq V_0$ in $U_{\ensuremath{I}}$, where $\ensuremath{I}\in Q_h$, is defined as
\[
\beta_{\ensuremath{I}}(X):=\frac{|\mathrm{Im}_\ensuremath{\mathcal{G}}(X,U_{\ensuremath{I}})|}{|X|}.
\]
If $\ensuremath{I}=\{i\}$, then we will use $\beta_i(X)$ to denote $\beta_{\{i\}}(X)$.
The following lemma relates the $\beta_{\ensuremath{I}}$ to the usual magnification ratio $\mu_i$.
\begin{lemma}
\label{beta ineq 1}
  Let $\ensuremath{\mathcal{G}}$ be a hypercube graph indexed by $Q_h$.
For any $\emptyset \neq X\subseteq V_0(\ensuremath{\mathcal{G}})$ we have
\[
\mu_i(\ensuremath{\mathcal{G}})\leq \sum_{|I|=i}\beta_{\ensuremath{I}}(X),
\]
with equality if and only if $X$ achieves $\mu_i(\ensuremath{\mathcal{G}})$ i.e., $\displaystyle \mu_i(\ensuremath{\mathcal{G}}) =   \frac{|\mathrm{Im}_\ensuremath{\mathcal{G}}(X,V_i)|}{|X|}$.
\end{lemma}
\begin{proof}
  By the definition of $\mu_i(\ensuremath{\mathcal{G}})$ we have
\[
\mu_i(\ensuremath{\mathcal{G}})\leq\frac{|\mathrm{Im}_\ensuremath{\mathcal{G}}(X,V_i)|}{|X|}
\]
with equality if and only if $X$ achieves $\mu_i(\ensuremath{\mathcal{G}})$.
Since $V_i$ is a disjoint union of the $U_{\ensuremath{I}}$ such that $|\ensuremath{I}|=i$, we have
\[
  \frac{|\mathrm{Im}_\ensuremath{\mathcal{G}}(X,V_i)|}{|X|} = \frac{|\biguplus_{|\ensuremath{I}|=i}\mathrm{Im}_\ensuremath{\mathcal{G}}(X,U_{\ensuremath{I}})|}{|X|}=\sum_{|\ensuremath{I}|=i}\frac{|\mathrm{Im}_\ensuremath{\mathcal{G}}(X,U_{\ensuremath{I}})|}{|X|}=\sum_{|\ensuremath{I}|=i}\beta_{\ensuremath{I}}(X).
\]
Combining these two equations yields the desired inequality.
\end{proof}
Later we will also need the following elementary identity, which asserts that the $\beta_i$ are multiplicative.
\begin{lemma}
\label{beta ineq 3}
  Let $\ensuremath{\mathcal{G}}^\prime, \ensuremath{\mathcal{G}}^{\prime\prime}$ be hypercube graphs indexed by $Q_h$ and $\ensuremath{\mathcal{G}} = \ensuremath{\mathcal{G}}^\prime \otimes \ensuremath{\mathcal{G}}^{\prime\prime}$. Then for all $i=1,\dots,h$ and $Z^\prime \subseteq V_0(\ensuremath{\mathcal{G}}^\prime)$, $Z^{\prime\prime} \subseteq V_0(\ensuremath{\mathcal{G}}^{\prime\prime})$ we have:
\[
\beta_i(Z^\prime \times Z^{\prime\prime}) = \beta_i(Z^\prime)  \beta_i(Z^{\prime\prime}).
\]
\end{lemma}
\begin{proof}
  We have $V_1(\ensuremath{\mathcal{G}}^\prime) = \biguplus_{i=1}^h U^\prime_{\{i\}}$ and $V_1(\ensuremath{\mathcal{G}}^{\prime\prime}) = \biguplus_{i=1}^h U^{\prime\prime}_{\{i\}}$. The way $\ensuremath{\mathcal{G}}$ is constructed gives $V_1(\ensuremath{\mathcal{G}}) = \biguplus_{i=1}^h (U^\prime_{\{i\}} \times U^{\prime\prime}_{\{i\}})$. Note that
\[
  \mathrm{Im}_{\ensuremath{\mathcal{G}}} (Z^\prime \times Z^{\prime\prime} , U^\prime_{\{i\}} \times U^{\prime\prime}_{\{i\}}  )  =   \mathrm{Im}_{\ensuremath{\mathcal{G}}^\prime} (Z^\prime  , U^\prime_{\{i\}})  \times \mathrm{Im}_{\ensuremath{\mathcal{G}}^{\prime\prime}} (Z^{\prime\prime}  , U^{\prime\prime}_{\{i\}}) .
\]
The claim follows by taking cardinalities:
\begin{align*}
\beta_i(Z^\prime \times Z^{\prime\prime})  &  = \frac{| \mathrm{Im}_{\ensuremath{\mathcal{G}}} (Z^\prime \times Z^{\prime\prime} , U^\prime_{\{i\}} \times U^{\prime\prime}_{\{i\}} )|}{|Z^\prime \times Z^{\prime\prime}|} \\
								&  = \frac{| \mathrm{Im}_{\ensuremath{\mathcal{G}}^\prime} (Z^\prime , U^\prime_{\{i\}})|}{|Z^\prime|}   \frac{| \mathrm{Im}_{\ensuremath{\mathcal{G}}^{\prime\prime}} ( Z^{\prime\prime} , U^{\prime\prime}_{\{i\}}  )|}{| Z^{\prime\prime}|}  \\
								& = \beta_i(Z^\prime)  \beta_i(Z^{\prime\prime}).
\end{align*}
\end{proof} 
Magnification ratio is multiplicative with respect to tensor product of layered graphs. However, for hypercube graphs this is only true for the top level magnification ratio, which is multiplicative for square commutative hypercube graphs. Square commutativity is not necessary, but it is sufficient (logically and for our purposes).
\begin{lemma}
\label{top level mag}
  Let $\ensuremath{\mathcal{G}}_1$ and $\ensuremath{\mathcal{G}}_2$ be square commutative hypercube graphs indexed by $Q_h$, and let $\ensuremath{\mathcal{G}}_3$ be the hypercube product $\ensuremath{\mathcal{G}}_1\otimes\ensuremath{\mathcal{G}}_2$.
Then $\mu_h(\ensuremath{\mathcal{G}}_3)=\mu_h(\ensuremath{\mathcal{G}}_1)\mu_h(\ensuremath{\mathcal{G}}_2)$.
\end{lemma}
\begin{proof}
For $i=1,2,3$, we will define an auxiliary layered graph $\hat{\ensuremath{\mathcal{G}}}_i$ as follows: $V_0(\hat{\ensuremath{\mathcal{G}}}_i):=V_0(\ensuremath{\mathcal{G}}_i)$, $V_1(\hat{\ensuremath{\mathcal{G}}}_i):=V_{h}(\ensuremath{\mathcal{G}}_i)$, and $(v,v^{\prime})\in E(\hat{\ensuremath{\mathcal{G}}}_i)$ if and only if there is a path from $v$ to $v^{\prime}$ in $\ensuremath{\mathcal{G}}_i$.
The proof rests on the following fact:
\begin{indexedclaim}
  $\hat{\ensuremath{\mathcal{G}}}_3 = \hat{\ensuremath{\mathcal{G}}}_1 \times \hat{\ensuremath{\mathcal{G}}}_2$.
\end{indexedclaim}
In words, $\hat{\ensuremath{\mathcal{G}}}_3$ is the directed layered tensor product of $\hat{\ensuremath{\mathcal{G}}}_1$ and $\hat{\ensuremath{\mathcal{G}}}_2$.
It should be noted here that this would not be the case if we were working with $i$th magnification ratios for  $1\leq i <h$, and that square commutativity is essential for our proof.
\begin{proof}[Proof of claim]
It suffices to show that for any pair of vertices $(u_0,v_0)$ in $V_0(\ensuremath{\mathcal{G}}_1)\times V_0(\ensuremath{\mathcal{G}}_2)$ and any pair of vertices $(u_h,v_h)$ in $V_h(\ensuremath{\mathcal{G}}_1)\times V_h(\ensuremath{\mathcal{G}}_2)$, we can find a sequence of index sets $\emptyset\to I_1\to\cdots\to I_h=\{1,\ldots,h\}$, and paths $u_0\to u_1\to\cdots\to u_h$  in $\ensuremath{\mathcal{G}}_1$ and $v_0\to v_1\to\cdots\to v_h$ in $\ensuremath{\mathcal{G}}_2$ such that $u_j\in U_{I_j}(\ensuremath{\mathcal{G}}_1)$ and $v_j\in U_{I_j}(\ensuremath{\mathcal{G}}_2)$.
This guarantees that the product path $(u_0,v_0)\to (u_1,v_1)\to\cdots\to (u_h,v_h)$ is contained in the hypercube product $\ensuremath{\mathcal{G}}_1\otimes\ensuremath{\mathcal{G}}_2$, hence the edge $(u_0,v_0)\to (u_h,v_h)$ is contained in $\hat{\ensuremath{\mathcal{G}}}_3$.

Let $u_0\to u_1\to\cdots\to u_h$ be any path in $\ensuremath{\mathcal{G}}_1$ from $u_0$ to $u_h$.
We will use square commutativity to show that there is a path $u_0\to u_1^{\prime}\to\cdots\to u_{h-1}^{\prime}\to u_h$ such that $u_j^{\prime}\in U_{\{1,\ldots,j\}}$.
Applying the same argument for a path $v_0\to v_1\to\cdots\to v_h$ will prove the claim.

For each $u_j$, let $I_j$ be the index set in $Q_h$ such that $u_j\in U_{I_j}(\ensuremath{\mathcal{G}}_1)$.
By definition, $u_j\to u_{j+1}$ only if there exists $i_{j+1}$ such that $I_{j+1}=I_j\cup\{i_{j+1}\}$.
Thus we may represent the sequence of index sets by a permutation:
\[
  \begin{pmatrix}
  1 & 2 &\cdots & h\\
  i_1& i_2&\cdots&i_h
\end{pmatrix}.
\]
Applying upward square commutativity to the sequence $I_{j-1}\to I_j\to I_{j+1}$ is equivalent to switching the pair $i_j$ and $i_{j+1}$. An example that illustrates this, is that by applying upward square commutativity to the layers $V_0$, $V_1$ and $V_2$ we transform the sequence $V_0 = U_\emptyset \to U_{\{i_1\}} \to U_{\{i_1,i_2\}}$ to $V_0 = U_\emptyset \to U_{\{i_2\}} \to U_{\{i_1,i_2\}}$ and so, in the permutation notation, we get 
\[
  \begin{pmatrix}
  1 & 2 & 3 & \cdots & h\\
  i_1& i_2  & i_3 & \cdots&i_h
\end{pmatrix}
\to 
\begin{pmatrix}
  1 & 2 & 3 & \cdots & h\\
  i_2& i_1& i_3 & \cdots&i_h
\end{pmatrix}.
\]
Thus by repeated application of upward square commutativity, we can find a path $u_0\to u_1^{\prime}\to u_2^{\prime}\to\cdots\to u_h$ such that $u_1^{\prime}\in U_{\{1\}}(\ensuremath{\mathcal{G}}_1)$.
Again by repeated application of square commutativity, we can find a path $u_0\to u_1^{\prime}\to u_2^{\prime \prime}\to u_3^{\prime \prime}\to\cdots\to u_h$ such that $u_2^{\prime \prime}\in U_{\{1,2\}}(\ensuremath{\mathcal{G}}_1)$, and so on.
\end{proof}

Now we continue with the proof of the lemma.
By definition of $\hat{\ensuremath{\mathcal{G}}}_i$, we have $\mu_1(\hat{\ensuremath{\mathcal{G}}}_i)=\mu_h(\ensuremath{\mathcal{G}}_i)$ for $i=1,2,3$.
Since $\hat{\ensuremath{\mathcal{G}}}_3$ is the layered product of $\hat{\ensuremath{\mathcal{G}}}_1$ and $\hat{\ensuremath{\mathcal{G}}}_2$, by the multiplicativity of magnification ratios of directed layered graphs (e.g. Theorem 7.1 in \cite{Nathanson1996}) we have $\mu_1(\hat{\ensuremath{\mathcal{G}}}_3)=\mu_1(\hat{\ensuremath{\mathcal{G}}}_1)\mu_1(\hat{\ensuremath{\mathcal{G}}}_2)$.
Thus $\mu_h(\ensuremath{\mathcal{G}}_3)=\mu_h(\ensuremath{\mathcal{G}}_1)\mu_h(\ensuremath{\mathcal{G}}_2)$, as desired.
\end{proof}
We are now ready to state and prove the theorem. 

\begin{theorem}[A Pl\"unnecke-type inequality for square commutative graphs]
  \label{Ruzsa analog}
Let $\ensuremath{\mathcal{G}}$ be a square commutative graph indexed by $Q_h$. Then for every $\emptyset \neq Z \subseteq V_0$ we have
\[
\mu_h(\ensuremath{\mathcal{G}}) \leq\beta_1(Z) \cdots \beta_h(Z).
\]
Moreover, 
\[
\mu_h(\ensuremath{\mathcal{G}}) \leq \left( \frac{\mu_1(\ensuremath{\mathcal{G}})}{h} \right)^h.
\]
\end{theorem}
\begin{proof}
As usual $\biguplus_{i=0}^h V_i$ is the vertex set of $\ensuremath{\mathcal{G}}$ and $V_1 = \biguplus_{i=1}^h U_{\{i\}}$.

$\ensuremath{\mathcal{G}}$ is a square commutative graph and so in particular is commutative. Applying Theorem \ref{Plunnecke} and Lemma  \ref{beta ineq 1} successively gives: 
 \[
  \mu_h(\ensuremath{\mathcal{G}}) \leq \mu_1(\ensuremath{\mathcal{G}})^h \leq \left( \sum_{i=1}^h \beta_i(Z) \right)^h,
\]
for all $\emptyset \neq Z \subseteq V_0.$
A first improvement is as follows.
\begin{indexedclaim}
\label{beta ineq 2}
For all $\emptyset \neq Z \subseteq V_0$, we have
\begin{align} \label{beta ineq 2 eqn}
  \mu_h(\ensuremath{\mathcal{G}}) \leq \left( \max_{1\leq i\leq h}\beta_i(Z)\right)^h.
\end{align}  
\end{indexedclaim}
To prove Claim \ref{beta ineq 2} we use the tensor product trick (\cite{Ruzsa2010}, see also \cite{TaoBlogTensor}). Let $n$ be any positive integer. We let $\ensuremath{\mathcal{G}}^n=\ensuremath{\mathcal{G}}\otimes\cdots\otimes\ensuremath{\mathcal{G}}$ denote the $n$-fold hypercube product of $\ensuremath{\mathcal{G}}$ with itself and $S^n \subseteq V_i(\ensuremath{\mathcal{G}}^n)$ the subset of $V_i(\ensuremath{\mathcal{G}}^n)$ that is precisely the $n$-fold product of $S$ with itself. 

By Lemma \ref{top level mag} and Theorem \ref{Plunnecke} \& Lemma \ref{SqComProd} we get that for all positive integers $n$
\[
 \mu_h(\ensuremath{\mathcal{G}})^n   =    \mu_h(\ensuremath{\mathcal{G}}^n)  \leq  \mu_1(\ensuremath{\mathcal{G}}^n)^h.
\]
By Lemma \ref{beta ineq 1} and Lemma \ref{beta ineq 3} we have $\mu_1(\ensuremath{\mathcal{G}}^n) \leq \sum_{i=1}^h \beta_i(Z^n) =  \sum_{i=1}^h \beta_i(Z)^n$ and so
\[
\mu_h(\ensuremath{\mathcal{G}}) \leq\left( \sum_{i=1}^h \beta_i(Z)^n \right)^{h/n}.
\]
Letting $n$ go to infinity proves Claim \ref{beta ineq 2}.

To deduce the first inequality in the statement of the theorem we use a trick of Ruzsa (e.g. \cite{Ruzsa2009}), which appears in his proof of Theorem \ref{Ruzsa} and is similar to that used in the deduction of Theorem \ref{Main} from Proposition \ref{Main equal}.

We begin by recalling that $Z$ is fixed. Let $T_1,\ldots, T_h$ be pairwise disjoint sets of generators of a free abelian group with identity $0$.
For now we leave $n_i=|T_i|$ undetermined, but note that they will depend on $Z$.

Let $\ensuremath{\mathcal{T}}$ denote the addition graph $\ensuremath{\mathcal{G}}_+(\{0\},T_1,\ldots,T_h)$ and let $\ensuremath{\mathcal{G}}^\prime=\ensuremath{\mathcal{G}}\otimes\ensuremath{\mathcal{T}}$. The subsets of $V_0(\ensuremath{\mathcal{G}}^\prime)$ are of the form $S\times \{0\}$ for $S\subseteq V_0.$ 

Combining Claim \ref{beta ineq 2} with Lemma \ref{beta ineq 3} gives 
\[
\mu_h(\ensuremath{\mathcal{G}}^\prime)  \leq   \left( \max_{1\leq i\leq h}\beta_i(Z\times\{0\})\right)^h  = \left( \max_{1\leq i\leq h}\beta_i(Z)   \beta_i(\{0\}) \right)^h = \left( \max_{1\leq i\leq h}\beta_i(Z)   n_i \right)^h .
\]

We now chose the value of the $n_i$. The $\beta_i(Z)$ are rational numbers so we set $\beta_i(Z) = p_i / q_i$ and $n=q_1 \cdots q_h$. By choosing $n_i= n \prod_{j\not=i}\beta_j(Z)$ we have $\beta_i(Z)n_i= n \prod_{\ell =1}^h n_\ell = \beta_j(Z)n_j$ for all $i,j=1,\ldots,n$.
Thus
\[
  \left(\max_{1\leq i\leq k}\beta_i(Z)   n_i\right)^h=\prod_{i=1}^h \beta_i(Z) n_i.
\]
On the other hand Lemma \ref{top level mag} gives
\[
 \mu_h(\ensuremath{\mathcal{G}}^\prime) = \mu_h(\ensuremath{\mathcal{G}})   \mu_h(\ensuremath{\mathcal{T}}) = \mu_h(\ensuremath{\mathcal{G}})   |T_1+\dots + T_h| = \mu_h(\ensuremath{\mathcal{G}})   n_1\dots n_h.
\]
Combining the above proves the first inequality in the statement of the theorem:
\[
\mu_h(\ensuremath{\mathcal{G}}) = \frac{\mu_h(\ensuremath{\mathcal{G}}^\prime)}{n_1\dots n_h} \leq \frac{\left( \max_{1\leq i\leq h}\beta_i(Z)   n_i \right)^h}{n_1\dots n_h} = \frac{ \prod_{i=1}^h \beta_i(Z) n_i}{n_1\dots n_h} = \prod_{i=1}^h \beta_i(Z).
\]
To get the second inequality in the statement of the theorem we first apply the arithmetic mean - geometric mean inequality and get 
\[
\mu_h(\ensuremath{\mathcal{G}}) \leq \beta_1(Z) \cdots \beta_h(Z) \leq \left( \frac{1}{h} \sum_{i=1}^h \beta_i(Z) \right)^h .
\]  
The last step is to let $\emptyset \neq X \subseteq V_0$ be the subset that achieves the first magnification ratio $\mu_1(\ensuremath{\mathcal{G}})$ i.e., $\displaystyle \mu_1(\ensuremath{\mathcal{G}}) =   \frac{|\mathrm{Im}_\ensuremath{\mathcal{G}}(X,V_1)|}{|X|}$.  Lemma \ref{beta ineq 1} gives $ \sum_{i=1}^h \beta_i(X) = \mu_1(\ensuremath{\mathcal{G}})$ and we are done.
\end{proof}
A couple remarks of some interest. 

Considering $\ensuremath{\mathcal{G}} = \ensuremath{\mathcal{G}}_+(\{0\}, T_1,\dots, T_h)$ as constructed above with $|T_1|=\dots=|T_h|$ shows that the upper bound cannot be trivially improved. 

Theorem \ref{Plunnecke} follows from Theorem \ref{Ruzsa analog}. Let $\ensuremath{\mathcal{G}}$ be a commutative graph with vertex set $V_0, V_1, \dots, V_h.$ We construct a hypercube graph $\ensuremath{\mathcal{H}}$ as follows: $U_\ensuremath{I} = V_{| \ensuremath{I} |}$ and for every $\ensuremath{I} \to \ensuremath{I}^{\prime}$, $u \in U_\ensuremath{I}$ and $v \in \ensuremath{I}^{\prime},$ $uv \in E(\ensuremath{\mathcal{H}})$ if and only if $uv \in E(\ensuremath{\mathcal{G}}).$

One may think of of the $i$th layer of $\ensuremath{\mathcal{H}}$ as consisting of $\tbinom{h}{i}$ copies of $V_i$ and the set of edges between $U_\ensuremath{I}$ and $U_{\ensuremath{I}^{\prime}}$ is a copy of the set of edges between $V_{|I|}$ and $V_{|I^{\prime}|}$, whenever $I \to I^{\prime}.$

A routine calculation confirms that $\ensuremath{\mathcal{H}}$ is square commutative, that $\mu_h(\ensuremath{\mathcal{H}}) = \mu_h(\ensuremath{\mathcal{G}})$ and that $\mu_1(\ensuremath{\mathcal{H}}) = h \mu_1(\ensuremath{\mathcal{G}}).$ Therefore,
\[
\mu_h(\ensuremath{\mathcal{G}}) = \mu_h(\ensuremath{\mathcal{H}}) \leq  \left( \frac{\mu_1(\ensuremath{\mathcal{H}})}{h} \right)^h =   \mu_1(\ensuremath{\mathcal{G}}) ^h.
\]

\subsection{A stronger Pl\"unnecke-type inequality for square commutative graphs}
\label{Strong Plunnecke-type}

Theorem \ref{Ruzsa analog} has one unsatisfactory aspect from a technical point of view: it does not provide any information on the subset $\emptyset \neq Z \subseteq V_0$ that achieves $\mu_h(\ensuremath{\mathcal{G}})$ i.e., the subset that satisfies $\displaystyle \mu_h(\ensuremath{\mathcal{G}}) =   \frac{|\mathrm{Im}_\ensuremath{\mathcal{G}}(Z,V_h)|}{|Z|}$. We strengthen Theorem \ref{Ruzsa analog} by proving that the subset $\emptyset \neq X \subseteq V_0$ that achieves $\mu_1(\ensuremath{\mathcal{G}})$ has restricted growth and in fact satisfies the bound given in Theorem \ref{Ruzsa analog}. A similar result was proved for commutative graphs in \cite{GPCaSu}.

\begin{theorem}\label{Stronger Ruzsa analog}
Let $\ensuremath{\mathcal{G}}$ be a square commutative graph with vertex set $V_0\cup\dots\cup V_h$. Suppose that $\emptyset \neq X \subseteq V_0$ achieves $\mu_1(\ensuremath{\mathcal{G}})$ i.e., $\displaystyle \mu_1(\ensuremath{\mathcal{G}}) =   \frac{|\mathrm{Im}_\ensuremath{\mathcal{G}}(X,V_1)|}{|X|}$. Then 
\[
|\mathrm{Im}_\ensuremath{\mathcal{G}}(X,V_h)| \leq \left( \frac{\mu_1(\ensuremath{\mathcal{G}})}{h}\right)^h |X|.
\]  
\end{theorem}
\begin{proof}
We work in the channel $\ensuremath{\mathcal{G}}^\prime = \ensuremath{\overline{\ensuremath{\mathcal{G}}}}(X,V_h)$ rather than the original square commutative graph. In this context we will prove that if $\ensuremath{\mathcal{G}}^\prime$ is a commutative graph with vertex set $V^\prime_0 \cup \dots \cup V^\prime_h$, which satisfies $\mu_1(\ensuremath{\mathcal{G}}^\prime)=|V^\prime_1| / |V^\prime_0|$, then 
\[
|V^\prime_h| \leq \left( \frac{\mu_1(\ensuremath{\mathcal{G}}^\prime)}{h}\right)^h  |V^\prime_0|.
\]

Suppose not. Let $\ensuremath{\mathcal{G}}^\prime$ be a counterexample where $|V^\prime_0|$ is minimal. Theorem \ref{Ruzsa analog} implies that the collection 
\[
\left\{ \emptyset \neq Z \subseteq V^\prime_0 : |\mathrm{Im}_{\ensuremath{\mathcal{G}}^\prime}(Z,V^\prime_h)| \leq  \left( \frac{\mu_1({\ensuremath{\mathcal{G}}^\prime})}{h}\right)^h |Z| \right\}
\]
is nonempty. 

Let $S\subsetneq V^\prime_0$ be a set of maximal cardinality in the collection ($S$ cannot equal $V_0^\prime$ as we have assumed that $\ensuremath{\mathcal{G}}^\prime$ is a counter example) and $\ensuremath{\mathcal{H}} = \ensuremath{\overline{\ensuremath{\mathcal{G}}^\prime}}(V^\prime_0 \setminus S , V^\prime_h \setminus \mathrm{Im}_{\ensuremath{\mathcal{G}}^\prime}(S,V^\prime_h))$. In words $\ensuremath{\mathcal{H}}$ is the channel consisting of paths in ${\ensuremath{\mathcal{G}}^\prime}$ that do not start in $S$ and do not end in its image in $V^\prime_h$.  Suppose that $W_0\cup W_1\cup\dots\cup W_h$ are the layers of $\ensuremath{\mathcal{H}}$. Observe also that for all $Z\subseteq W_0$ and all $i=1,\dots,h$ we have $\mathrm{Im}_\ensuremath{\mathcal{H}}(Z,W_i) = \mathrm{Im}_{\ensuremath{\mathcal{G}}^\prime}(Z,W_i)$.

$W_1$ does not intersect $\mathrm{Im}_{\ensuremath{\mathcal{G}}^\prime}(S,V^\prime_1)$ as there would
then exist a path in $\ensuremath{\mathcal{H}}$ leading to $\mathrm{Im}_{\ensuremath{\mathcal{G}}^\prime}(S,V^\prime_h)$. We therefore
have  
\begin{align*}
|W_1| &  \leq|V^\prime_1|-|\mathrm{Im}_{\ensuremath{\mathcal{G}}^\prime}(S,V^\prime_1)| \\ 
	  &  \leq |V^\prime_1| - \mu_1({\ensuremath{\mathcal{G}}^\prime}) |S| \\
	  & = |V^\prime_1| -\frac{|V^\prime_1|}{|V^\prime_0|}  |S| \\
	  & = |V^\prime_1|   \frac{|V^\prime_0|-|S|}{|V^\prime_0|} \\
	  & = |W_0|  \frac{|V^\prime_1|}{|V^\prime_0|} ,
\end{align*}
as $|W_0| = |V^\prime_0| - |S|$.
Consequently
\begin{equation}\label{mu_1 of S complement in Stronger Plunnecke}
 \mu_1(\ensuremath{\mathcal{H}})\leq \frac{|W_1|}{|W_0|}\leq \frac{|V^\prime_1|}{|V^\prime_0|}=\mu_1({\ensuremath{\mathcal{G}}^\prime}).
\end{equation}
Let $\emptyset \neq T\subseteq W_0$ be any subset that satisfies $|\mathrm{Im}_\ensuremath{\mathcal{H}}(T,W_1)| =
\mu_1(\ensuremath{\mathcal{H}})  |T|$. Let us get a lower bound on $|\mathrm{Im}_\ensuremath{\mathcal{H}}(T,W_h)|$. We
know from the maximality of $|S|$ that
\begin{align*}
\left(\frac{\mu_1({\ensuremath{\mathcal{G}}^\prime})}{h}\right)^{h}  |S\cup T| & <  |\mathrm{Im}_{\ensuremath{\mathcal{G}}^\prime}(S\cup T,V^\prime_h)| \\
                   			  			 & =   |\mathrm{Im}_{\ensuremath{\mathcal{G}}^\prime}(S,V^\prime_h)| + |\mathrm{Im}_{\ensuremath{\mathcal{G}}^\prime}(T,V^\prime_h) \setminus \mathrm{Im}_{\ensuremath{\mathcal{G}}^\prime}(S,V^\prime_h)|  \\
                   			   			 & =    |\mathrm{Im}_{\ensuremath{\mathcal{G}}^\prime}(S,V^\prime_h)|+|\mathrm{Im}_\ensuremath{\mathcal{H}}(T,W_h)|  \\
		 	                     			 & \leq   \left(\frac{\mu_1({\ensuremath{\mathcal{G}}^\prime})}{h}\right)^h   |S| + |\mathrm{Im}_\ensuremath{\mathcal{H}}(T,W_h)|.
\end{align*}
This implies
\begin{equation}\label{mu_h of T in Stronger Plunnecke}
 |\mathrm{Im}_\ensuremath{\mathcal{H}}(T,W_h)| > \left( \frac{\mu_1({\ensuremath{\mathcal{G}}^\prime})}{h}\right)^h  |T|.
\end{equation}
Finally we consider $\ensuremath{\mathcal{H}}^\prime = \ensuremath{\overline{\ensuremath{\mathcal{H}}}}(T,W_h)$, the channel consisting of all paths in $\ensuremath{\mathcal{H}}$ starting at $T$. $\ensuremath{\mathcal{H}}^\prime$ is a square commutative graph with layers $T_0 \cup\dots\cup T_h$ and magnification ratio $\mu_1(\ensuremath{\mathcal{H}}^\prime) = \mu_1(\ensuremath{\mathcal{H}})$. By inequalities \eqref{mu_h of T in Stronger Plunnecke} and \eqref{mu_1 of S complement in Stronger Plunnecke}  we get:
\begin{align*}
|T_h| & =  |\mathrm{Im}_{\ensuremath{\mathcal{H}}^\prime} (T,W_h)|  =  |\mathrm{Im}_\ensuremath{\mathcal{H}}(T,W_h)| \\ 
          & > \left(\frac{\mu_1({\ensuremath{\mathcal{G}}^\prime})}{h}\right)^h |T|  \geq  \left(\frac{\mu_1(\ensuremath{\mathcal{H}})}{h}\right)^h |T_0| \\
          & = \left(\frac{\mu_1(\ensuremath{\mathcal{H}}^\prime)}{h}\right)^h |T_0|.
\end{align*}
Thus $\ensuremath{\mathcal{H}}^\prime$ is another counterexample. However, $|T_0|=|T| \leq |W_0| = |V^\prime_0 \setminus S| < |V^\prime_0|$, which contradicts the minimality of $|V^\prime_0|$.
\end{proof}

\subsection{Application to sumsets with a component removed}
\label{Ruzsa restricted}

Our final task is to deduce from Theorem \ref{Stronger Ruzsa analog} the upper bound on sumsets with a component removed, which was used in Section \ref{Proof of main} .
\begin{corollary}\label{Restricted sumsets}
Let $h$ be a positive integer. Suppose that $A, B_1,\dots,B_h$ are finite sets in a commutative group and $E\subseteq A$ a subset of $A$. 

If $\emptyset\neq X \subseteq A\setminus E$ is a subset of $A\setminus E$ that minimises the quantity
\[
\sum_{i=1}^h \frac{ |(Z+B_i)\setminus (E+B_i)|}{|Z|}
\]
over all non-empty subsets $\emptyset \neq Z \subseteq A\setminus E$, then
\[
|(X+B_1+\dots+ B_h)   \setminus (E+B_1+\cdots + B_h)| \leq \mu^h   |X|,
\]
where
\[
\mu: = \mu(X) = \frac{1}{h}  \sum_{i=1}^h \frac{ |(X+B_i)\setminus (E+B_i)|}{|X|} .
\]
\end{corollary}
\begin{proof}
We work in to the hypercube graph $\ensuremath{\mathcal{G}}$ indexed by $Q_h$ with vertex set given by $U_{\ensuremath{I}} = (A + \sum_{i\in I} B_i) \setminus  (E + \sum_{i\in I} B_i)$; and edge set determined as follows: an edge exists between $u\in V_I$ and $v\in V_{I\cup\{j\}}$ if $v-u \in B_j.$

$\ensuremath{\mathcal{G}}$ is square commutative by Lemma \ref{ChSqCom}, because it is precisely the channel 
\[
\ensuremath{\overline{\ensuremath{\mathcal{G}}}}\left(A\setminus E  ,  (A + \sum_{i=1}^h B_i) \setminus  (E + \sum_{i=1}^h B_i)\right)
\]
in the square commutative addition graph $\ensuremath{\mathcal{G}}_+(A,B_1,\dots,B_h)$.

Identifying $Z\subseteq A\setminus E$ with the corresponding subset of $V_0(\ensuremath{\mathcal{G}})$ gives
\[
\sum_{i=1}^h \frac{ |(Z+B_i) \setminus (E+B_i)|}{|Z|} = \frac{|\mathrm{Im}_\ensuremath{\mathcal{G}}(Z, V_1)|}{|Z|}.
\]
In particular the defining property of $X$ implies that $X$ achieves $\mu_1(\ensuremath{\mathcal{G}})$ and so 
\begin{align*}
\frac{\mu_1(\ensuremath{\mathcal{G}})}{h} & = \frac{1}{h}   \frac{|\mathrm{Im}_\ensuremath{\mathcal{G}}(X, V_1)|}{|X|} \\
							        & =  \frac{1}{h}  \sum_{i=1}^h \frac{ |\mathrm{Im}_\ensuremath{\mathcal{G}}(X, U_{\{i\}})| }{|X|}  \\
							        & = \frac{1}{h}  \sum_{i=1}^h \frac{ |(X+B_i)\setminus (E+B_i)| }{|X|} \\
							        & = \mu.
\end{align*}

The condition in Theorem \ref{Stronger Ruzsa analog} is satisfied and so  
\[
|(X+B_1+\dots+ B_h)   \setminus (E+B_1+\cdots + B_h)| = \mathrm{Im}_\ensuremath{\mathcal{G}}(X, V_h)   \leq \left( \frac{\mu_1(\ensuremath{\mathcal{G}})}{h}  \right)^h = \mu^h   |X|,
\]
as claimed.
\end{proof}

\subsection*{Acknowledgements}
The authors would like to thank the referee for a careful reading of the manuscript and helpful suggestions; and Dimitris Koukoulopoulos for spotting an omission  in the statement of Lemma~\ref{top level mag}. The second author would like to thank the University of Rochester for their funding during his postdoc there where the bulk of this project was completed. The third author would like to thank Imre Ruzsa for his help in simplifying some of the proofs in the related paper \cite{GPCaSu}, which made the proof of the main result of the present paper simpler than it would have otherwise been.

\phantomsection

\textsc{Department Mathematics, University of Rochester, Rochester, NY, USA}

\textit{Email addresses}: murphy@math.rochester.edu,   palsson@math.rochester.edu,  giorgis@cantab.net

\end{document}